\documentclass[11pt]{amsart}

\usepackage[a4paper,margin=1in]{geometry}
\usepackage{amsmath,amssymb,amsthm,mathtools}
\usepackage{mathrsfs}
\usepackage{enumitem}
\usepackage{hyperref}
\usepackage{xcolor}

\hypersetup{
  colorlinks=true,
  linkcolor=blue!50!black,
  citecolor=blue!50!black,
  urlcolor=blue!50!black
}

\numberwithin{equation}{section}

\newtheorem{theorem}{Theorem}[section]
\newtheorem{lemma}[theorem]{Lemma}
\newtheorem{proposition}[theorem]{Proposition}
\newtheorem{corollary}[theorem]{Corollary}
\newtheorem{remark}[theorem]{Remark}
\newtheorem{definition}[theorem]{Definition}
\newtheorem{example}[theorem]{Example}

\newcommand{\K}{\mathbb{K}}
\newcommand{\E}{\mathbf{E}}
\newcommand{\RB}{\mathrm{RB}}
\newcommand{\LRB}{\mathrm{LRB}}

\newcommand{\Der}{\operatorname{Der}}
\newcommand{\Aut}{\operatorname{Aut}}
\newcommand{\End}{\operatorname{End}}
\newcommand{\id}{\operatorname{id}}

\title[Local Rota--Baxter operators of weight zero]{Local Rota--Baxter operators of weight zero on nilpotent evolution algebras with maximal nilindex}

\author{Izzat Qaralleh}
\address{Izzat Qaralleh\\
Department of Mathematics\\
Faculty of Science, Tafila Technical
University\\
Tafila, Jordan}
\email{{\tt izzat\_math@yahoo.com}}

\author{Farrukh Mukhamedov}
\address{Farrukh Mukhamedov\\
 Department of Mathematical Sciences\\
College of Science, The United Arab Emirates University\\
P.O. Box, 15551, Al Ain\\
Abu Dhabi, UAE} \email{{\tt far75m@gmail.com} {\tt
farrukh.m@uaeu.ac.ae}}

\author{Otabek Khakimov}
\address{Otabek Khakimov\\
V.I.Romanovskiy Institute  of Mathematics,  Uzbekistan Academy of Sciences\\
Tashkent, Uzbekistan}
\email{{\tt khakimovo86@gmail.com}}

\date{\today}

\begin{document}

\begin{abstract}
Let $\E$ be an $n$-dimensional nilpotent evolution algebra of maximal nilindex over a field of characteristic zero. The Rota--Baxter operators of weights zero and one on such algebras were recently classified. In this paper we investigate the local analogue of the weight-zero case. We prove that every local Rota--Baxter operator of weight zero has a rigid diagonal part: the nonzero diagonal coefficients, when they occur, begin after the obstruction index determined by the structural matrix of $\E$ and form a final geometric string governed by one scalar. In contrast, the last-column coefficients are arbitrary. This gives an explicit description of the class $\LRB_0(\E)$ and shows that it is generally strictly larger than $\RB_0(\E)$. We also prove that the corresponding $s$-local notion produces no new operators. The classification is further interpreted as a finite union of quasi-affine strata, yielding a dimension comparison between $\LRB_0(\E)$ and $\RB_0(\E)$. Finally, we study ordinary weight-zero Rota--Baxter operators commuting with derivations and automorphisms and describe the resulting conditions in terms of the directed graph associated with $\E$.
\end{abstract}


\maketitle

\section{Introduction}

Rota--Baxter operators were introduced by Baxter in probability theory \cite{Baxter} and were soon developed from several complementary points of view by Atkinson and Rota \cite{Atkinson,RotaI,RotaII}.  They are now regarded as an algebraic abstraction of integral-type operators: in weight zero the defining identity is the algebraic form of integration by parts.  This simple principle has proved remarkably robust.  Rota--Baxter operators occur in algebraic combinatorics, dendriform and pre-Lie type structures, the algebraic Birkhoff decomposition in perturbative quantum field theory, Yang--Baxter type equations, post-Lie theory, group and Lie-theoretic integration problems, and, more recently, vertex operator algebra theory; see, for instance, \cite{Aguiar,BaiGuoLiuWang,BaiGuoNi,BaiGuoShengTang,EbrahimiFardGuoKreimer,Guo,GuoLangSheng}.  This breadth explains the continuing actuality of the subject: Rota--Baxter operators provide a common operator-theoretic language for splitting, factorization, and integration phenomena across different algebraic categories.
Recently, in \cite{Bar22,Bar23} it has been established a connection between skew braces and Rota--Baxter groups.  

Evolution algebras form another active class of nonassociative algebras.  They were introduced in connection with non-Mendelian genetics and stochastic processes \cite{TianVojtechovsky,TianBook}; their defining feature is the existence of a natural basis $\{{\bf e}_1,\ldots,{\bf e}_n\}$ for which distinct basis vectors multiply to zero.  This special multiplication rule makes evolution algebras simultaneously concrete and subtle.  On the one hand, many structural questions can be translated into matrix or directed-graph language.  On the other hand, standard tools from associative or Lie theory are not directly available.  Consequently, questions about nilpotency, ideals, derivations, automorphisms, and classification have generated a substantial literature; see, among others, \cite{CasasLadraOmirovRozikov,ElduqueLabra,ElduqueLabraNilpotent,ElduqueLabraAutGraphsHegaziAbdelwahab,MKQ20,DerAut}.

Nilpotent evolution algebras of maximal nilindex occupy an extremal position in this theory \cite{CasasLadraOmirovRozikov,MKQ24}.  Their natural basis can be chosen in a triangular normal form, so their multiplication consists of a directed chain together with possible higher jumps.  This makes them a particularly suitable testing ground for explicit operator classifications: the algebra is rigid enough to allow complete formulas, but still rich enough for the structural constants and the associated directed graph to influence the answer.  In particular, the obstruction index $i_0$ appearing below measures the last place where a higher structural coefficient occurs, and it is precisely this index that controls the possible onset of nonzero diagonal coefficients of Rota--Baxter operators \cite{QMK-RB}.

The ordinary Rota--Baxter operators of weights $0$ and $1$ on nilpotent evolution algebras of maximal nilindex were recently classified in \cite{QMK-RB}.  In weight zero the classification is especially rigid: an ordinary operator is upper triangular, all off-diagonal entries vanish except possibly those in the last column, and every nonzero diagonal part is a final geometric string determined by one scalar.  This rigidity is a strong global consequence of the quadratic Rota--Baxter identity.  It therefore leads naturally to the following locality problem: which parts of this rigidity remain valid if the Rota--Baxter identity is imposed only pointwise through ordinary Rota--Baxter witnesses? 

This question fits into the broader philosophy of local and $2$-local maps.  Local derivations and local automorphisms were initiated in operator-algebraic settings by Kadison and by Larson--Sourour \cite{Kadison,LarsonSourour}, and the $2$-local viewpoint was developed further by \v{S}emrl \cite{Semrl}; see also the survey \cite{AyupovKudaybergenovPeralta}.  The central theme in this area is to determine whether pointwise agreement with a prescribed class of operators forces a map to be global, or whether genuinely new local objects appear.  For Rota--Baxter operators this problem is particularly delicate, because the identity is nonlinear and involves the multiplication of images.  Thus the local version is not merely a formal weakening: it tests which constraints are forced by individual vectors and which constraints require simultaneous compatibility for all pairs of vectors.

The novelty of the present paper is to give a complete answer to this locality problem in the weight-zero case for nilpotent evolution algebras of maximal nilindex.  We prove that a local Rota--Baxter operator is not arbitrary: its diagonal part is as rigid as in the ordinary classification.  More precisely, if a nonzero diagonal part occurs, then it must start after the obstruction index $i_0$ and must have the form
\[
2^{n-i}\alpha
\]
on a final segment.  At the same time, the local theory differs genuinely from the ordinary theory: the last-column coefficients are free. Hence, the local class is usually strictly larger than the ordinary class,
\[
\RB_0(\E)\subsetneq \LRB_0(\E),
\]
and the difference is completely explicit.  This identifies exactly which part of the ordinary Rota--Baxter rigidity is local and which part is global.

A second contribution is the comparison between local and $s$-local versions of the same notion.  We show that the $s$-local condition produces no additional operators.  Thus, for this class of algebras and in weight zero, pointwise locality and finite-point locality lead to the same operator family.  This phenomenon is useful because it separates the genuinely new local freedom from artifacts of the definition.

The classification also has a geometric and parameter-space interpretation.  The set $\LRB_0(\E)$ decomposes into a finite union of strata: an affine annihilator-valued stratum and several quasi-affine strata indexed by the first basis vector at which the final diagonal string begins.  This yields a transparent dimension comparison between the ordinary and local spaces and shows that the local parameter space has one additional degree of freedom.  In this way the paper connects an operator identity with the algebraic geometry of parameter spaces.

Finally, we study ordinary weight-zero Rota--Baxter operators that commute with derivations and automorphisms of $\E$.  These results complement the local classification by describing how Rota--Baxter thresholds interact with the symmetries and infinitesimal symmetries of the algebra.  The resulting equations admit a graph-theoretic interpretation, in the spirit of the directed-graph approach to evolution algebras \cite{ElduqueLabra,ElduqueLabraAutGraphs}.  This provides an additional reason why the maximal-nilindex case is important: the directed graph records the obstruction index, the terminal annihilator, and the compatibility conditions for commuting operators in a single combinatorial picture.

The paper is organized as follows.  In Section~\ref{sec:preliminaries}, we recall the normal form of nilpotent evolution algebras of maximal nilindex, the relevant notation, and the ordinary weight-zero classification from \cite{QMK-RB}.  Section~\ref{sec:main} contains the local theory and the proof of the classification theorem.  In Section~\ref{sec:examples}, we give explicit low-dimensional examples, including local operators that are not ordinary Rota--Baxter operators.  Section~\ref{sec:parameter-space} describes the parameter-space structure of $\LRB_0(\E)$ and compares it with $\RB_0(\E)$.  Section~\ref{sec:commuting} studies commuting Rota--Baxter operators with derivations and automorphisms and gives a graph-theoretic interpretation of the resulting conditions.  The final section summarizes the main conclusions and indicates possible directions for further work.
\section{Preliminaries}\label{sec:preliminaries}

Throughout the paper, $\K$ denotes a field of characteristic zero. All vector spaces and algebras are taken over $\K$. We assume $n\ge 3$ throughout; this is the range in which the normal form below is used.

\subsection{Nilpotent evolution algebras of maximal nilindex}

We briefly recall the standard setup; see \cite{CasasLadraOmirovRozikov, HegaziAbdelwahab, QMK-RB}.

\begin{definition}
An algebra $\E$ over $\K$ is called an \emph{evolution algebra} if it admits a basis $\{{\bf e}_1,\dots,{\bf e}_n\}$ such that
\[
{\bf e}_i {\bf e}_j=0 \quad (i\neq j), \qquad
{\bf e}_i^2=\sum_{k=1}^n a_{ik} {\bf e}_k.
\]
Such a basis is called a \emph{natural basis}.
\end{definition}

We restrict attention to nilpotent evolution algebras of maximal nilindex. By Lemma~2.3 of \cite{QMK-RB}, such an algebra admits a natural basis $\{{\bf e}_1,\dots,{\bf e}_n\}$ satisfying
\begin{equation}\label{eq:natural-form}
{\bf e}_i^2=
\begin{cases}
{\bf e}_{i+1}+\displaystyle\sum_{j=i+2}^{n-1} a_{ij} {\bf e}_j, & 1\le i\le n-3,\\[1ex]
{\bf e}_{n-1}, & i=n-2,\\
{\bf e}_n, & i=n-1,\\
{\bf 0}, & i=n.
\end{cases}
\end{equation}
As in \cite{QMK-RB}, we write
\begin{equation}\label{eq:IA}
\mathcal I_A=\{(i,j): i+1<j<n,\ a_{ij}\neq 0\},
\end{equation}
and define
\[
i_0=
\begin{cases}
0, & \mathcal I_A=\varnothing,\\
\max\{i:(i,j)\in\mathcal I_A \text{ for some }j\}, & \mathcal I_A\neq\varnothing.
\end{cases}
\]

\subsection{Ordinary and local Rota--Baxter operators}

\begin{definition}
A linear map $R:\E\to \E$ is called a \emph{Rota--Baxter operator of weight $0$} if
\begin{equation}\label{eq:RB0}
R({\bf x})R({\bf y})=R(R({\bf x}){\bf y}+{\bf x}R({\bf y})), \qquad {\bf x},{\bf y}\in \E.
\end{equation}
We denote the set of all such operators by $\RB_0(\E)$.
\end{definition}

\begin{definition}
A linear map $T:\E\to \E$ is called a \emph{local Rota--Baxter operator of weight $0$} if for every ${\bf x}\in \E$ there exists an operator $R_{\bf x}\in \RB_0(\E)$ such that
\[
T({\bf x})=R_{\bf x}({\bf x}).
\]
We denote the set of all such operators by $\LRB_0(\E)$.
\end{definition}

Clearly,
\[
\RB_0(\E)\subseteq \LRB_0(\E).
\]

\subsection{Previously known weight-zero classification}

We now record the result from \cite[Theorem~4.5]{QMK-RB} that will be used throughout the paper.

\begin{theorem}\label{thm:ordinary-classification}
Let $\E$ be as in \eqref{eq:natural-form}. A linear map $R:\E\to \E$ belongs to $\RB_0(\E)$ if and only if exactly one of the following alternatives holds.
\begin{enumerate}[label=\textup{(\roman*)}]
\item The image of $R$ is contained in $\K{\bf e}_n$, equivalently there exist scalars $\beta_1,\dots,\beta_n\in\K$ such that
\[
R({\bf e}_i)=\beta_i{\bf e}_n\qquad (1\le i\le n).
\]
\item There exist an index $k\in\{i_0+1,\dots,n-1\}$, a scalar $\alpha\in\K^\times$, and scalars $\beta_1,\dots,\beta_k\in\K$ such that
\[
R({\bf e}_i)=
\begin{cases}
\beta_i {\bf e}_n, & 1\le i<k,\\
2^{n-k}\alpha\,{\bf e}_k+\beta_k {\bf e}_n, & i=k,\\
2^{n-i}\alpha\,{\bf e}_i, & k<i\le n.
\end{cases}
\]
\end{enumerate}
\end{theorem}

\begin{remark}\label{rem:ordinary-shape}
The theorem shows, in particular, that every ordinary weight-zero Rota--Baxter operator is upper triangular, and that for each basis vector ${\bf e}_i$ one has
\[
R({\bf e}_i)\in
\begin{cases}
\K {\bf e}_n, & i\leq i_0,\\
\K {\bf e}_i+\K {\bf e}_n, & i_0<i<n,\\
\K {\bf e}_n, & i=n.
\end{cases}
\]
Moreover, once a nonzero diagonal coefficient appears, all later diagonal coefficients are forced and are determined by the same scalar $\alpha$.
\end{remark}

\section{Local weight-zero Rota--Baxter operators}\label{sec:main}

In this section we determine $\LRB_0(\E)$ completely.

\begin{lemma}\label{lem:basis-images}
Let $T\in \LRB_0(\E)$. Then:
\begin{enumerate}[label=\textup{(\alph*)}]
\item for every $i\le i_0$, one has $T({\bf e}_i)\in \K {\bf e}_n$;
\item for every $i_0<i<n$, one has $T({\bf e}_i)\in \K {\bf e}_i+\K {\bf e}_n$;
\item one has $T({\bf e}_n)\in \K {\bf e}_n$.
\end{enumerate}
\end{lemma}

\begin{proof}
Fix $i\in\{1,\dots,n\}$. Since $T$ is local, there exists $R_{{\bf e}_i}\in \RB_0(\E)$ such that
\[
T({\bf e}_i)=R_{{\bf e}_i}({\bf e}_i).
\]
Now apply Theorem~\ref{thm:ordinary-classification} to the ordinary operator $R_{{\bf e}_i}$.

If $i\le i_0$, then the threshold $k$ from Theorem~\ref{thm:ordinary-classification} always satisfies $k>i_0\ge i$, hence $i<k$, and therefore
\[
R_{{\bf e}_i}({\bf e}_i)\in \K {\bf e}_n.
\]
Thus $T({\bf e}_i)\in \K {\bf e}_n$.

If $i_0<i<n$, then there are two possibilities in Theorem~\ref{thm:ordinary-classification}: either $i<k$, in which case $R_{{\bf e}_i}({\bf e}_i)\in \K {\bf e}_n$, or $i\ge k$, in which case
\[
R_{{\bf e}_i}({\bf e}_i)\in \K {\bf e}_i+\K {\bf e}_n.
\]
Hence $T({\bf e}_i)\in \K {\bf e}_i+\K {\bf e}_n$.

Finally, if $i=n$, then Theorem~\ref{thm:ordinary-classification} gives
\[
R_{{\bf e}_n}({\bf e}_n)\in \K {\bf e}_n.
\]
Therefore $T({\bf e}_n)\in \K {\bf e}_n$.
\end{proof}

Consequently, every local operator has matrix form
\begin{equation}\label{eq:T-shape-1}
T({\bf e}_i)=
\begin{cases}
c_i {\bf e}_n, & 1\le i\le i_0,\\
d_i {\bf e}_i+c_i {\bf e}_n, & i_0<i<n,\\
c_n {\bf e}_n, & i=n,
\end{cases}
\end{equation}
for suitable scalars $c_1,\dots,c_n,d_{i_0+1},\dots,d_{n-1}\in \K$.

The next result shows that the diagonal coefficients $\{d_i\}$ satisfy a very strong restriction.

\begin{proposition}\label{prop:final-segment}
Let $T\in \LRB_0(\E)$ and write $T$ as in \eqref{eq:T-shape-1}. Then the set
\[
J=\{i\in\{i_0+1,\dots,n-1\}: d_i\neq 0\}
\]
is either empty or of the form
\[
J=\{k,k+1,\dots,n-1\}
\]
for some $k\in\{i_0+1,\dots,n-1\}$.
\end{proposition}

\begin{proof}
Assume $J\neq\varnothing$, and let $i\in J$. We claim that every $j>i$ also belongs to $J$.

Suppose, to the contrary, that there exists $j>i$ with $d_j=0$. Consider the vector
\[
{\bf x}={\bf e}_i+{\bf e}_j.
\]
Since $T$ is local, there exists $R_{\bf x}\in \RB_0(\E)$ such that
\[
T({\bf x})=R_{\bf x}({\bf x}).
\]
By \eqref{eq:T-shape-1},
\[
T({\bf x})=d_i {\bf e}_i + (c_i+c_j){\bf e}_n,
\]
because $d_j=0$.

Now apply Theorem~\ref{thm:ordinary-classification} to $R_{\bf x}$. Since the coefficient of ${\bf e}_i$ in $R_{\bf x}({\bf x})$ is nonzero, the threshold $k_{\bf x}$ of $R_{\bf x}$ must satisfy $k_{\bf x}\le i$. But then $j>i\ge k_{\bf x}$, so the same theorem implies that the coefficient of ${\bf e}_j$ in $R_{\bf x}({\bf x})$ is
\[
2^{\,n-j}\alpha_{\bf x},
\]
where $\alpha_{\bf x}\neq 0$ because the coefficient of ${\bf e}_i$ is nonzero. Thus the ${\bf e}_j$-coefficient of $R_{\bf x}({\bf x})$ is also nonzero, contradicting the formula for $T({\bf x})$ above. Therefore $d_j\neq 0$ for all $j>i$.

This proves that $J$ is a final segment.
\end{proof}

\begin{proposition}\label{prop:common-alpha}
Let $T\in \LRB_0(\E)$ and assume that
\[
J=\{k,k+1,\dots,n-1\}
\]
is nonempty. Then there exists a scalar $\alpha\in \K\setminus\{0\}$ such that
\[
d_i=2^{n-i}\alpha,\qquad k\le i\le n-1.
\]
\end{proposition}

\begin{proof}
Fix $i,j\in J$ with $i<j$. Consider
\[
{\bf x}={\bf e}_i+{\bf e}_j.
\]
Since $T$ is local, there exists $R_{\bf x}\in \RB_0(\E)$ such that $T({\bf x})=R_{\bf x}({\bf x})$.

By \eqref{eq:T-shape-1},
\[
T({\bf x})=d_i {\bf e}_i+d_j {\bf e}_j+(c_i+c_j){\bf e}_n.
\]
Both ${\bf e}_i$- and ${\bf e}_j$-coefficients are nonzero. Therefore, by Theorem~\ref{thm:ordinary-classification}, the threshold $k_{\bf x}$ of $R_{\bf x}$ must satisfy $k_{\bf x}\le i$, and the diagonal part of $R_{\bf x}$ is determined by one scalar $\alpha_{\bf x}$. Hence
\[
d_i=2^{n-i}\alpha_{\bf x},\qquad d_j=2^{n-j}\alpha_{\bf x}.
\]
It follows that
\[
2^{i-n}d_i=2^{j-n}d_j.
\]
Since $i,j\in J$ were arbitrary, the quantity $2^{i-n}d_i$ is independent of $i\in J$. Let
\[
\alpha:=2^{k-n}d_k.
\]
Then
\[
d_i=2^{n-i}\alpha,\qquad k\le i\le n-1.
\]
Because $d_k\neq 0$, we also have $\alpha\neq 0$.
\end{proof}

We can now state the main theorem.

\begin{theorem}\label{thm:main-local}
Let $\E$ be an $n$-dimensional nilpotent evolution algebra of maximal nilindex in the normal form \eqref{eq:natural-form}. A linear operator $T:\E\to \E$ belongs to $\LRB_0(\E)$ if and only if exactly one of the following alternatives holds.
\begin{enumerate}[label=\textup{(\roman*)}]
\item The operator is annihilator-valued: there exist scalars $c_1,\dots,c_n\in\K$ such that
\begin{equation}\label{eq:local-zero-diagonal}
T({\bf e}_i)=c_i{\bf e}_n\qquad (1\le i\le n).
\end{equation}
\item The diagonal part is nonzero and is a final geometric string: there exist an index
\[
k\in\{i_0+1,\dots,n-1\},
\]
a scalar $\alpha\in\K^\times$, and scalars $c_1,\dots,c_n\in\K$ such that
\begin{equation}\label{eq:local-classification}
T({\bf e}_i)=
\begin{cases}
c_i {\bf e}_n, & 1\le i<k,\\
2^{n-i}\alpha\,{\bf e}_i+c_i {\bf e}_n, & k\le i<n,\\
c_n {\bf e}_n, & i=n.
\end{cases}
\end{equation}
\end{enumerate}
Equivalently, every $T\in\LRB_0(\E)$ has no matrix entries outside the diagonal and the last column; its diagonal entries on ${\bf e}_1,\dots,{\bf e}_{n-1}$ are either all zero, or they are
\[
0,\dots,0,2^{n-k}\alpha,2^{n-k-1}\alpha,\dots,2\alpha
\]
for a unique $k\in\{i_0+1,\dots,n-1\}$ and a unique $\alpha\in\K^\times$. In both cases the last-column coefficients are arbitrary.
\end{theorem}

\begin{proof}
We first prove necessity.

Let $T\in \LRB_0(\E)$. By Lemma~\ref{lem:basis-images}, $T$ has the form \eqref{eq:T-shape-1}. If all diagonal coefficients $d_i$ vanish for $i_0<i<n$, then $T$ has the form \eqref{eq:local-zero-diagonal}.

Assume now that some $d_i$ is nonzero. By Proposition~\ref{prop:final-segment}, the set of indices with nonzero diagonal coefficient is a final segment
\[
J=\{k,k+1,\dots,n-1\}
\]
for some $k\in\{i_0+1,\dots,n-1\}$. By Proposition~\ref{prop:common-alpha}, there exists $\alpha\in\K^\times$ such that
\[
d_i=2^{n-i}\alpha,\qquad k\le i\le n-1.
\]
Thus $T$ has the form \eqref{eq:local-classification}.

We now prove sufficiency.

If $T$ has the annihilator-valued form \eqref{eq:local-zero-diagonal}, then $T\in\RB_0(\E)$ by the annihilator-valued alternative in Theorem~\ref{thm:ordinary-classification}. Hence $T\in\LRB_0(\E)$.

Assume that $T$ has the form \eqref{eq:local-classification}. Let
\[
{\bf x}=\sum_{i=1}^n x_i {\bf e}_i\in \E.
\]
We construct an ordinary operator $R_{\bf x}\in\RB_0(\E)$ such that
\[
R_{\bf x}({\bf x})=T({\bf x}).
\]

There are two cases.

\smallskip
\noindent\emph{Case 1:} $x_k=x_{k+1}=\cdots=x_{n-1}=0$.

In this case,
\[
T({\bf x})=\left(\sum_{i=1}^{n-1} c_i x_i + c_n x_n\right){\bf e}_n.
\]
Define $R_{\bf x}$ by
\[
R_{\bf x}({\bf e}_i)=c_i{\bf e}_n\quad (1\le i<n),
\qquad
R_{\bf x}({\bf e}_n)=c_n{\bf e}_n.
\]
By Theorem~\ref{thm:ordinary-classification}, this is an ordinary weight-zero Rota--Baxter operator, and it plainly satisfies $R_{\bf x}({\bf x})=T({\bf x})$.

\smallskip
\noindent\emph{Case 2:} at least one of $x_k,\dots,x_{n-1}$ is nonzero.

Let
\[
s=\min\{i\in\{k,\dots,n-1\}: x_i\neq 0\}.
\]
We define an ordinary operator $R_{\bf x}$ with nonzero threshold $s$. For $i<s$, set
\[
\beta_i=c_i.
\]
Next, define
\begin{equation}\label{eq:beta-s}
\beta_s=
\frac{c_s x_s+c_{s+1}x_{s+1}+\cdots+c_{n-1}x_{n-1}+(c_n-\alpha)x_n}{x_s}.
\end{equation}
Since $x_s\neq 0$, this is well-defined. With the same scalar $\alpha$ as in \eqref{eq:local-classification}, Theorem~\ref{thm:ordinary-classification} gives an ordinary operator $R_{\bf x}\in\RB_0(\E)$ satisfying
\[
R_{\bf x}({\bf e}_i)=
\begin{cases}
\beta_i {\bf e}_n, & i<s,\\
2^{n-s}\alpha\,{\bf e}_s+\beta_s {\bf e}_n, & i=s,\\
2^{n-i}\alpha\,{\bf e}_i, & s<i\le n.
\end{cases}
\]

Since $x_k=\cdots=x_{s-1}=0$, we obtain
\[
R_{\bf x}({\bf x})
=
\sum_{i=s}^{n-1}2^{n-i}\alpha x_i{\bf e}_i
+
\left(\sum_{i=1}^{s-1}c_i x_i+\beta_sx_s+\alpha x_n\right){\bf e}_n.
\]
By the choice of $\beta_s$ in \eqref{eq:beta-s}, the ${\bf e}_n$-coefficient is
\[
\sum_{i=1}^{n-1} c_i x_i+c_nx_n.
\]
Therefore
\[
R_{\bf x}({\bf x})
=
\sum_{i=s}^{n-1}2^{n-i}\alpha x_i{\bf e}_i
+
\left(\sum_{i=1}^{n-1} c_i x_i+c_nx_n\right){\bf e}_n
=T({\bf x}),
\]
because $x_k=\cdots=x_{s-1}=0$. Hence $T$ is local.
\end{proof}
\begin{corollary}\label{cor:ordinary-inclusion}
Every ordinary weight-zero Rota--Baxter operator is local. Moreover, let $T\in\LRB_0(\E)$.
\begin{enumerate}[label=\textup{(\roman*)}]
\item If $T(\E)\subseteq\K{\bf e}_n$, then $T\in\RB_0(\E)$.
\item If $T$ has nonzero diagonal part and is written as in \eqref{eq:local-classification}, then
\[
T\in\RB_0(\E)
\]
if and only if
\[
c_i=0\qquad (k<i<n),
\qquad\text{and}\qquad
c_n=\alpha.
\]
\end{enumerate}
\end{corollary}

\begin{proof}
The inclusion $\RB_0(\E)\subseteq\LRB_0(\E)$ is immediate from the definition.

If $T(\E)\subseteq\K{\bf e}_n$, then $T$ is ordinary by the annihilator-valued alternative in Theorem~\ref{thm:ordinary-classification}. This proves (i).

Assume now that $T$ has nonzero diagonal part and is written as in \eqref{eq:local-classification}. Comparing this form with the nonzero-diagonal alternative in Theorem~\ref{thm:ordinary-classification}, the coefficient $c_k$ is allowed, but after the threshold no last-column terms are allowed and the image of ${\bf e}_n$ is forced to be $\alpha{\bf e}_n$. Hence the necessary and sufficient conditions are precisely $c_i=0$ for $k<i<n$ and $c_n=\alpha$.
\end{proof}
\begin{remark}
Theorem~\ref{thm:main-local} shows that the diagonal rigidity found in \cite{QMK-RB} survives in the local setting, but the last column loses rigidity completely. This is the precise point at which the local class becomes larger than the ordinary one.
\end{remark}

From the last result, we infer that the class of local Rota--Baxter operators is larger than the class of Rota--Baxter operators. Therefore, it is natural to investigate local Rota--Baxter operators further.

Namely, a linear map $K:\E\to\E$ is called an \textit{$s$-local Rota--Baxter operator} of weight $0$ if for every ${\bf x}\in\E$ there exists a local Rota--Baxter operator $T_{\bf x}$ of weight $0$ on $\E$ such that
\[
K({\bf x})=T_{\bf x}({\bf x}).
\]

Clearly, every local Rota--Baxter operator is $s$-local.
\begin{theorem}\label{thm:s-local-classification}
Let $\E$ be an $n$-dimensional nilpotent evolution algebra of maximal nilindex in the normal form \eqref{eq:natural-form}. Then every $s$-local Rota--Baxter operator of weight $0$ on $\E$ is a local Rota--Baxter operator of weight $0$.
\end{theorem}

\begin{proof}
Assume that $K$ is an $s$-local Rota--Baxter operator of weight $0$. For each ${\bf x}\in\E$ there exists $T_{\bf x}\in\LRB_0(\E)$ such that
\[
K({\bf x})=T_{\bf x}({\bf x}).
\]

Applying Theorem~\ref{thm:main-local} to the operators attached to the basis vectors, we get
\[
K({\bf e}_i)=
\begin{cases}
c_i{\bf e}_n, & 1\le i\le i_0,\\
d_i{\bf e}_i+c_i{\bf e}_n, & i_0<i<n,\\
c_n{\bf e}_n, & i=n,
\end{cases}
\]
for suitable scalars $c_i,d_i\in\K$.

Let
\[
J=\{i\in\{i_0+1,\dots,n-1\}:d_i\neq0\}.
\]
We first show that $J$ is either empty or a final segment. Suppose that $i\in J$ and that $j>i$ with $j<n$. If $d_j=0$, then for ${\bf x}={\bf e}_i+{\bf e}_j$ we have
\[
K({\bf x})=d_i{\bf e}_i+(c_i+c_j){\bf e}_n.
\]
Choose $T_{\bf x}\in\LRB_0(\E)$ with $K({\bf x})=T_{\bf x}({\bf x})$. Since the coefficient of ${\bf e}_i$ is nonzero, Theorem~\ref{thm:main-local} forces the nonzero diagonal string of $T_{\bf x}$ to start at or before $i$. Hence the coefficient of ${\bf e}_j$ in $T_{\bf x}({\bf x})$ is also nonzero, a contradiction. Thus $d_j\neq0$, and $J$ is a final segment.

If $J=\varnothing$, then $K(\E)\subseteq\K{\bf e}_n$, so $K\in\LRB_0(\E)$ by Theorem~\ref{thm:main-local}.

Assume now that
\[
J=\{k,k+1,\dots,n-1\}
\]
for some $k\in\{i_0+1,\dots,n-1\}$. For any $i,j\in J$ with $i<j$, apply the $s$-local property to ${\bf x}={\bf e}_i+{\bf e}_j$. Theorem~\ref{thm:main-local} gives a single scalar $\alpha_{\bf x}\neq0$ such that
\[
d_i=2^{n-i}\alpha_{\bf x},
\qquad
d_j=2^{n-j}\alpha_{\bf x}.
\]
Therefore $2^{i-n}d_i=2^{j-n}d_j$ for all $i,j\in J$. Taking $\alpha=2^{k-n}d_k$, we obtain
\[
d_i=2^{n-i}\alpha\qquad (k\le i<n),
\]
with $\alpha\in\K^\times$. Hence $K$ has the nonzero-diagonal form \eqref{eq:local-classification}. By Theorem~\ref{thm:main-local}, $K\in\LRB_0(\E)$.
\end{proof}
\section{Examples and counterexamples}\label{sec:examples}

In this section we illustrate Theorem~\ref{thm:main-local} on several low-dimensional nilpotent evolution algebras of maximal nilindex. In each case, we verify explicitly that the given operators satisfy the structural description of local Rota--Baxter operators of weight zero, and we show that they are not ordinary Rota--Baxter operators.

\begin{example}
Let $\E$ be the three-dimensional evolution algebra with natural basis $\{{\bf e}_1,{\bf e}_2,{\bf e}_3\}$ and multiplication
\[
{\bf e}_1^2 = {\bf e}_2, \qquad {\bf e}_2^2 = {\bf e}_3, \qquad {\bf e}_3^2 ={\bf 0}.
\]
Then $\mathcal I_A = \varnothing$, so $i_0 = 0$.

By Theorem~\ref{thm:main-local}, a linear operator $T:\E \to \E$ belongs to $\LRB_0(\E)$ if and only if either $T(\E)\subseteq\K{\bf e}_3$, or there exist $k\in\{1,2\}$, $\alpha\in\K^\times$, and $c_1,c_2,c_3\in\K$ such that
\[
T({\bf e}_i)=
\begin{cases}
c_i {\bf e}_3, & i<k,\\
2^{3-i}\alpha\, {\bf e}_i + c_i {\bf e}_3, & k \le i < 3,\\
c_3 {\bf e}_3, & i=3.
\end{cases}
\]

Consider the operator
\[
T({\bf e}_1)={\bf 0}, \qquad T({\bf e}_2)=2{\bf e}_2+5{\bf e}_3, \qquad T({\bf e}_3)=7{\bf e}_3.
\]
This operator is of the form above with $k=2$ and $\alpha=1$, since $2^{3-2}=2$. Hence $T \in \LRB_0(\E)$ by Theorem~\ref{thm:main-local}.

We show that $T \notin \RB_0(\E)$. Indeed, if $T$ were an ordinary Rota--Baxter operator, then by Theorem~\ref{thm:ordinary-classification} with $k=2$ and $\alpha=1$, one would necessarily have
\[
T({\bf e}_3)=\alpha {\bf e}_3 = {\bf e}_3,
\]
which contradicts $T({\bf e}_3)=7{\bf e}_3$. Therefore,
\[
T \in \LRB_0(\E)\setminus \RB_0(\E).
\]
\end{example}

\begin{example}
Let $\E$ be the four-dimensional evolution algebra with natural basis $\{{\bf e}_1,{\bf e}_2,{\bf e}_3,{\bf e}_4\}$ and multiplication
\[
{\bf e}_1^2 = {\bf e}_2, \qquad {\bf e}_2^2 = {\bf e}_3, \qquad {\bf e}_3^2 = {\bf e}_4, \qquad {\bf e}_4^2 = {\bf 0}.
\]
Again $\mathcal I_A = \varnothing$, so $i_0 = 0$.

By Theorem~\ref{thm:main-local}, for $n=4$ the diagonal coefficients are $2^{4-i}\alpha$. Hence they equal $4\alpha$ for $i=2$ and $2\alpha$ for $i=3$.

Consider the linear map
\[
T({\bf e}_1)=3{\bf e}_4, \quad T({\bf e}_2)=4{\bf e}_2+{\bf e}_4, \quad T({\bf e}_3)=2{\bf e}_3-6{\bf e}_4, \quad T({\bf e}_4)=10{\bf e}_4.
\]
This operator satisfies the form in Theorem~\ref{thm:main-local} with $k=2$ and $\alpha=1$, since $2^{4-2}=4$ and $2^{4-3}=2$. Therefore $T \in \LRB_0(\E)$.

However, $T$ is not an ordinary Rota--Baxter operator. Indeed, if $T \in \RB_0(\E)$, then by Theorem~\ref{thm:ordinary-classification} with $k=2$ and $\alpha=1$, one must have
\[
T({\bf e}_3)=2{\bf e}_3, \qquad T({\bf e}_4)={\bf e}_4,
\]
which contradicts the given expressions. Hence
\[
T \in \LRB_0(\E)\setminus \RB_0(\E).
\]
\end{example}

\begin{example}
Let $\E$ be the four-dimensional evolution algebra with natural basis $\{{\bf e}_1,{\bf e}_2,{\bf e}_3,{\bf e}_4\}$ and multiplication
\[
{\bf e}_1^2 = {\bf e}_2 + a {\bf e}_3, \qquad {\bf e}_2^2 = {\bf e}_3, \qquad {\bf e}_3^2 = {\bf e}_4, \qquad {\bf e}_4^2 ={\bf 0},
\]
where $a \ne 0$. Then $(1,3) \in\mathcal I_A$, hence $\mathcal I_A \ne \varnothing$ and $i_0 = 1$.

By Theorem~\ref{thm:main-local}, every local Rota--Baxter operator of weight zero either satisfies $T(\E)\subseteq\K{\bf e}_4$, or has the following nonzero-diagonal form. One has
\[
T({\bf e}_1) = c_1 {\bf e}_4,
\]
and for $k \in \{2,3\}$,
\[
T({\bf e}_2)=
\begin{cases}
c_2 {\bf e}_4, & k=3,\\
4\alpha\, {\bf e}_2 + c_2 {\bf e}_4, & k=2,
\end{cases}
\]
\[
T({\bf e}_3)=2\alpha\,{\bf e}_3+c_3{\bf e}_4,
\qquad
T({\bf e}_4)=c_4 {\bf e}_4.
\]

In particular, $T({\bf e}_1)\in \K {\bf e}_4$, so no ${\bf e}_1$-component may appear in the image of ${\bf e}_1$.

Consider the operator corresponding to $k=2$ and $\alpha=1$:
\[
T({\bf e}_1)=0, \quad T({\bf e}_2)=4{\bf e}_2+5{\bf e}_4, \quad T({\bf e}_3)=2{\bf e}_3+7{\bf e}_4, \quad T({\bf e}_4)=9{\bf e}_4.
\]
This operator satisfies the form in Theorem~\ref{thm:main-local} and therefore belongs to $\LRB_0(\E)$.

We show that $T \notin \RB_0(\E)$. If $T$ were an ordinary Rota--Baxter operator, then by Theorem~\ref{thm:ordinary-classification} with $k=2$ and $\alpha=1$, one would necessarily have
\[
T({\bf e}_3)=2{\bf e}_3, \qquad T({\bf e}_4)={\bf e}_4,
\]
which contradicts the given expressions. Hence
\[
T \in \LRB_0(\E)\setminus \RB_0(\E).
\]

This example shows that even when $\mathcal I_A \ne \varnothing$, the local and ordinary classes do not coincide. The role of $\mathcal I_A$ is to determine the obstruction index $i_0$, which controls the earliest position at which the diagonal geometric progression may begin.
\end{example}

\section[Dimension and structure of local parameter spaces]{Dimension and structure of \texorpdfstring{$\LRB_0(\E)$}{LRB0(E)} as a parameter space}\label{sec:parameter-space}

The classification in Theorem~\ref{thm:main-local} admits a natural interpretation in terms of a finite union of quasi-affine families.

\begin{proposition}
Let $\E$ be an $n$-dimensional nilpotent evolution algebra of maximal nilindex in the normal form \eqref{eq:natural-form}, and let $i_0$ be defined by \eqref{eq:IA}. Then the set $\LRB_0(\E)$ of local Rota--Baxter operators of weight zero admits the following decomposition:
\begin{equation}\label{eq:LRB-stratification}
\LRB_0(\E)=\mathcal{V}_0\;\sqcup\;\bigcup_{k=i_0+1}^{n-1}\mathcal{V}_k,
\end{equation}
where \(\mathcal V_0\) consists of the maps \(T_{\mathbf c}\), with
\(\mathbf c=(c_1,\dots,c_n)\in\K^n\), defined by
\[
T_{\mathbf c}({\bf e}_i)=c_i{\bf e}_n\qquad(1\le i\le n),
\]
and, for each \(k\in\{i_0+1,\dots,n-1\}\), the stratum \(\mathcal V_k\) consists
of the maps \(T_{\alpha,\mathbf c}^{(k)}\), with
\(\alpha\in\K^\times\) and \(\mathbf c=(c_1,\dots,c_n)\in\K^n\), defined by
\[
T_{\alpha,\mathbf c}^{(k)}({\bf e}_i)=
\begin{cases}
c_i {\bf e}_n, & 1\le i<k,\\[1mm]
2^{n-i}\alpha {\bf e}_i+c_i {\bf e}_n, & k\le i<n,\\[1mm]
c_n {\bf e}_n, & i=n.
\end{cases}
\]

Moreover:
\begin{enumerate}
\item $\mathcal{V}_0$ is an affine subspace of $\End(\E)$ of dimension $n$, canonically isomorphic to $\K^n$;
\item for each $k\in\{i_0+1,\dots,n-1\}$, the set $\mathcal{V}_k$ is a quasi-affine irreducible subset of $\End(\E)$, canonically isomorphic to $\K^n\times \K^\times$, and hence
\[
\dim \mathcal{V}_k = n+1;
\]
\item the Zariski closure of each $\mathcal{V}_k$ is
\[
\overline{\mathcal{V}_k}=\mathcal{V}_k\cup \mathcal{V}_0,
\]
so $\mathcal{V}_0$ is the common boundary of all nonzero-diagonal strata;
\item consequently,
\[
\dim \LRB_0(\E)=n+1.
\]
\end{enumerate}
\end{proposition}

\begin{proof}
By Theorem~\ref{thm:main-local}, every operator $T\in \LRB_0(\E)$ is of the form
\begin{equation}\label{eq:LRB-normal-form}
T({\bf e}_i)=
\begin{cases}
c_i {\bf e}_n, & 1\le i<k,\\[1mm]
2^{\,n-i}\alpha\, {\bf e}_i+c_i {\bf e}_n, & k\le i<n,\\[1mm]
c_n {\bf e}_n, & i=n,
\end{cases}
\end{equation}
for some $k\in\{i_0+1,\dots,n-1\}$ and some scalars $\alpha,c_1,\dots,c_n\in \K$.

If $\alpha=0$, then \eqref{eq:LRB-normal-form} reduces to
\[
T({\bf e}_i)=c_i {\bf e}_n \qquad (1\le i\le n),
\]
independently of the choice of $k$. Thus all such operators form precisely the family $\mathcal{V}_0$.

If $\alpha\neq 0$, then the first index at which a nonzero diagonal coefficient appears is uniquely determined; namely, it is exactly $k$, because
\[
T({\bf e}_i)\in \K {\bf e}_n \quad (i<k),
\qquad
T({\bf e}_i)\notin \K {\bf e}_n \quad (k\le i<n).
\]
Hence every local Rota--Baxter operator with nonzero diagonal part belongs to a unique $\mathcal{V}_k$, where $k\in\{i_0+1,\dots,n-1\}$. This proves the disjoint decomposition \eqref{eq:LRB-stratification}.

The set $\mathcal{V}_0$ is clearly an affine subspace, with free parameters $c_1,\dots,c_n$, so $\mathcal{V}_0\cong \K^n$ and $\dim \mathcal{V}_0=n$.

Fix $k\in\{i_0+1,\dots,n-1\}$. The assignment
\[
(\alpha,c_1,\dots,c_n)\longmapsto T
\]
with $\alpha\in \K^\times$ and $c_1,\dots,c_n\in \K$, where $T$ is given by \eqref{eq:LRB-normal-form}, is injective and its image is exactly $\mathcal{V}_k$. Therefore
\[
\mathcal{V}_k\cong \K^\times\times \K^n,
\]
so $\mathcal{V}_k$ is irreducible and quasi-affine of dimension $n+1$.

Its Zariski closure is obtained by allowing $\alpha=0$, which gives precisely the operators in $\mathcal{V}_0$. Hence
\[
\overline{\mathcal{V}_k}=\mathcal{V}_k\cup \mathcal{V}_0.
\]

Finally, since $\LRB_0(\E)$ is the finite union of the strata $\mathcal{V}_0,\mathcal{V}_{i_0+1},\dots,\mathcal{V}_{n-1}$, its dimension is the maximum of their dimensions, namely
\[
\dim \LRB_0(\E)=n+1.
\]
\end{proof}

\begin{remark}
The proposition shows that $\LRB_0(\E)$ is naturally stratified by the first index $k$ at which the diagonal geometric progression begins. The stratum $\mathcal{V}_0$ corresponds to operators whose image is contained in $\K{\bf e}_n$, while each $\mathcal{V}_k$ for $k\le n-1$ consists of operators whose diagonal part begins exactly at ${\bf e}_k$.

In particular, $\LRB_0(\E)$ is generally \emph{not} a vector subspace of $\End(\E)$: it is a finite union of quasi-affine pieces, all of dimension at most $n+1$, glued along the common affine subspace $\mathcal{V}_0$.
\end{remark}

\begin{corollary}
The ordinary Rota--Baxter operators form the subset
\[
\RB_0(\E)
=
\mathcal{V}_0^{\,\mathrm{ord}}
\;\sqcup\;
\bigcup_{k=i_0+1}^{n-1}\mathcal{W}_k,
\]
where \(\mathcal V_0^{\mathrm{ord}}=\mathcal V_0\), and, for each
\(k\in\{i_0+1,\dots,n-1\}\), the stratum \(\mathcal W_k\) consists of the maps
\(S_{\alpha,\mathbf c}^{(k)}\), with \(\alpha\in\K^\times\) and
\(\mathbf c=(c_1,\dots,c_k)\in\K^k\), defined by
\[
S_{\alpha,\mathbf c}^{(k)}({\bf e}_i)=
\begin{cases}
c_i {\bf e}_n, & 1\le i<k,\\[1mm]
2^{n-k}\alpha {\bf e}_k+c_k {\bf e}_n, & i=k,\\[1mm]
2^{n-i}\alpha {\bf e}_i, & k<i<n,\\[1mm]
\alpha {\bf e}_n, & i=n.
\end{cases}
\]

Hence
\[
\dim \RB_0(\E)=\max\{n,\;k+1: i_0+1\le k\le n-1\}=n,
\]
whereas
\[
\dim \LRB_0(\E)=n+1.
\]
Thus the local parameter space is one dimension larger than the ordinary one.
\end{corollary}

\begin{proof}
This follows from Theorem~\ref{thm:ordinary-classification}. The annihilator-valued ordinary family is exactly $\mathcal V_0$, since every operator with image contained in $\K{\bf e}_n$ is ordinary. Hence it has dimension $n$.

For a fixed nonzero threshold $k\le n-1$, the ordinary family is parametrized by $\alpha\in\K^\times$ and $c_1,\dots,c_k\in\K$, hence has dimension $k+1$. The largest value is obtained at $k=n-1$, where it equals $n$. Therefore $\dim\RB_0(\E)=n$, while the preceding proposition gives $\dim\LRB_0(\E)=n+1$.
\end{proof}

\section{Commuting Rota--Baxter operators}\label{sec:commuting}

In this section we use the convention
\[
[S,T]=ST-TS
\]
for the commutator of two linear endomorphisms.  We keep the reduced normal form
\eqref{eq:natural-form}.  In a more general triangular presentation of nilpotent
maximal-nilindex evolution algebras, additional coefficients in the \({\bf e}_n\)-column
may occur; see \cite{DerAut}.  In the reduced normal form used here those coefficients
are zero.  Consequently, the general derivation and automorphism formulas specialize to
the following simple form.

\begin{proposition}\label{prop:der-auto-forms}
Let \(\E\) be given by \eqref{eq:natural-form}.
\begin{enumerate}[label=\textup{(\alph*)}]
\item If \(\mathcal I_A\neq\varnothing\), then every derivation \(D\in\Der(\E)\) has the form
\[
D({\bf e}_1)=\nu{\bf e}_n,
\qquad
D({\bf e}_i)={\bf 0}\quad (2\le i\le n)
\]
for some \(\nu\in\K\).

If \(\mathcal I_A=\varnothing\), then every derivation \(D\in\Der(\E)\) has the form
\[
D({\bf e}_1)=\mu{\bf e}_1+\nu{\bf e}_n,
\qquad
D({\bf e}_i)=2^{i-1}\mu{\bf e}_i\quad (2\le i\le n)
\]
for some \(\mu,\nu\in\K\).

\item If \(\phi\in\Aut(\E)\), then there exist \(\rho\in\K^\times\) and \(\gamma\in\K\) such that
\[
\phi({\bf e}_1)=\rho{\bf e}_1+\gamma{\bf e}_n,
\qquad
\phi({\bf e}_i)=\rho^{2^{i-1}}{\bf e}_i\quad (2\le i\le n).
\]
If \(\mathcal I_A\neq\varnothing\), then
\[
\rho^\eta=1,
\qquad
\eta=\gcd\{2^{j-1}-2^i:(i,j)\in\mathcal I_A\}.
\]
If \(\mathcal I_A=\varnothing\), there is no further restriction on \(\rho\) besides
\(\rho\ne0\).
\end{enumerate}
\end{proposition}

\begin{proof}
This is precisely the specialization of the derivation and automorphism formulas from
\cite[Theorems~3.2 and~5.1]{DerAut} to the reduced normal form \eqref{eq:natural-form}.
Indeed, the possible recursive sink terms in the unreduced triangular form are multiples
of coefficients of \({\bf e}_n\) in the products \({\bf e}_i^2\) with \(i\le n-2\); these
coefficients vanish in \eqref{eq:natural-form}.  Thus only the free first-row sink
coefficient remains.
\end{proof}

For later use we also isolate the ordinary weight-zero classification with notation
adapted to commutator computations.

\begin{proposition}\label{prop:RB0-form}
A linear operator \(R\) belongs to \(\RB_0(\E)\) if and only if exactly one of the following alternatives holds.
\begin{enumerate}[label=\textup{(\roman*)}]
\item \emph{Annihilator-valued case.}  There exist
\(\beta_1,\dots,\beta_{n-1},\lambda\in\K\) such that
\[
R({\bf e}_i)=\beta_i{\bf e}_n\quad (1\le i<n),
\qquad
R({\bf e}_n)=\lambda{\bf e}_n.
\]

\item \emph{Nonzero threshold case.}  There exist
\[
k\in\{i_0+1,\dots,n-1\},
\qquad
\lambda\in\K^\times,
\qquad
\beta_1,\dots,\beta_k\in\K,
\]
such that
\[
R({\bf e}_i)=
\begin{cases}
\beta_i{\bf e}_n, & 1\le i<k,\\[1mm]
2^{n-k}\lambda{\bf e}_k+\beta_k{\bf e}_n, & i=k,\\[1mm]
2^{n-i}\lambda{\bf e}_i, & k<i\le n.
\end{cases}
\]
\end{enumerate}
\end{proposition}

\begin{proof}
This is Theorem~\ref{thm:ordinary-classification}, with \(\lambda\) denoting the coefficient
of \({\bf e}_n\) in \(R({\bf e}_n)\).
\end{proof}

\begin{theorem}\label{thm:RB0-comm-der}
Let \(D\in\Der(\E)\), and let \(R\in\RB_0(\E)\) be written as in
Proposition~\ref{prop:RB0-form}.

\begin{enumerate}[label=\textup{(\arabic*)}]
\item Suppose \(\mathcal I_A\neq\varnothing\).  Write
\[
D({\bf e}_1)=\nu{\bf e}_n,
\qquad
D({\bf e}_i)={\bf 0}\quad(2\le i\le n).
\]
Then
\[
[R,D]=0 \quad\Longleftrightarrow\quad \nu\lambda=0.
\]
Thus, if \(\nu=0\), every ordinary weight-zero Rota--Baxter operator commutes with
\(D\); if \(\nu\ne0\), the commuting operators are exactly the annihilator-valued
operators with \(R({\bf e}_n)=0\).

\item Suppose \(\mathcal I_A=\varnothing\).  Write
\[
D({\bf e}_1)=\mu{\bf e}_1+\nu{\bf e}_n,
\qquad
D({\bf e}_i)=2^{i-1}\mu{\bf e}_i\quad(2\le i\le n),
\]
and put
\[
\tau_1=\nu,
\qquad
\tau_i=0\quad(2\le i<n).
\]
If \(R\) is annihilator-valued, say
\[
R({\bf e}_i)=\beta_i{\bf e}_n\quad(1\le i<n),
\qquad R({\bf e}_n)=\lambda{\bf e}_n,
\]
then \([R,D]=0\) if and only if
\begin{equation}\label{eq:comm-der-ann}
\beta_i(2^{i-1}-2^{n-1})\mu+\tau_i\lambda=0
\qquad(1\le i<n).
\end{equation}

If \(R\) has nonzero threshold \(k<n\), then \([R,D]=0\) if and only if
\begin{align}
\beta_i(2^{i-1}-2^{n-1})\mu+\tau_i\lambda&=0,
&&1\le i<k, \label{eq:comm-der-before}\\
\beta_k(2^{k-1}-2^{n-1})\mu+(1-2^{n-k})\tau_k\lambda&=0.&& \label{eq:comm-der-threshold}
\end{align}
There are no further conditions for \(k<i<n\); for those indices the commutator
vanishes automatically in the reduced normal form.
\end{enumerate}
\end{theorem}

\begin{proof}
The case \(\mathcal I_A\neq\varnothing\) follows immediately from
Proposition~\ref{prop:der-auto-forms}.  Since \(D({\bf e}_i)=0\) for \(i>1\) and
\(D({\bf e}_n)=0\), we have \([R,D]({\bf e}_i)=0\) for \(i>1\), while
\[
[R,D]({\bf e}_1)=R(\nu{\bf e}_n)-D(R({\bf e}_1))=\nu\lambda{\bf e}_n.
\]
Hence the condition is exactly \(\nu\lambda=0\).

Assume now that \(\mathcal I_A=\varnothing\).  For \(1\le i<n\) we may write
\[
D({\bf e}_i)=2^{i-1}\mu{\bf e}_i+\tau_i{\bf e}_n,
\]
where \(\tau_1=\nu\) and \(\tau_i=0\) for \(2\le i<n\).  Also
\(D({\bf e}_n)=2^{n-1}\mu{\bf e}_n\).

If \(R\) is annihilator-valued and \(i<n\), then
\[
RD({\bf e}_i)=\bigl(2^{i-1}\mu\beta_i+\tau_i\lambda\bigr){\bf e}_n,
\qquad
DR({\bf e}_i)=2^{n-1}\mu\beta_i{\bf e}_n.
\]
Thus \([R,D]({\bf e}_i)=0\) is exactly \eqref{eq:comm-der-ann}; the vector
\({\bf e}_n\) gives no condition.

Now suppose that \(R\) has nonzero threshold \(k\).  For \(i<k\) the same computation gives
\eqref{eq:comm-der-before}.  At the threshold,
\[
R({\bf e}_k)=2^{n-k}\lambda{\bf e}_k+\beta_k{\bf e}_n,
\]
and hence
\[
RD({\bf e}_k)=2^{n-1}\mu\lambda{\bf e}_k+
\bigl(2^{k-1}\mu\beta_k+\tau_k\lambda\bigr){\bf e}_n,
\]
whereas
\[
DR({\bf e}_k)=2^{n-1}\mu\lambda{\bf e}_k+
\bigl(2^{n-k}\lambda\tau_k+2^{n-1}\mu\beta_k\bigr){\bf e}_n.
\]
Their equality is equivalent to \eqref{eq:comm-der-threshold}.  Finally, for
\(k<i<n\),
\[
[R,D]({\bf e}_i)=(1-2^{n-i})\tau_i\lambda{\bf e}_n={\bf 0},
\]
because \(\tau_i=0\) on the tail.  This proves the theorem.
\end{proof}

\begin{theorem}\label{thm:RB0-comm-auto}
Let \(\phi\in\Aut(\E)\), and let \(R\in\RB_0(\E)\) be written as in
Proposition~\ref{prop:RB0-form}.  Write
\[
\phi({\bf e}_i)=\delta_i{\bf e}_i+\tau_i{\bf e}_n\quad(1\le i<n),
\qquad
\phi({\bf e}_n)=\delta_n{\bf e}_n,
\]
where
\[
\delta_i=\rho^{2^{i-1}}\quad(1\le i\le n),
\qquad
\tau_1=\gamma,
\qquad
\tau_i=0\quad(2\le i<n)
\]
as in Proposition~\ref{prop:der-auto-forms}.  Then the following hold.

\begin{enumerate}[label=\textup{(\roman*)}]
\item If \(R\) is annihilator-valued, then \([R,\phi]=0\) if and only if
\begin{equation}\label{eq:comm-auto-ann}
\beta_i(\delta_i-\delta_n)+\tau_i\lambda=0
\qquad(1\le i<n).
\end{equation}

\item If \(R\) has nonzero threshold \(k<n\), then \([R,\phi]=0\) if and only if
\begin{align}
\beta_i(\delta_i-\delta_n)+\tau_i\lambda&=0,
&&1\le i<k, \label{eq:comm-auto-before}\\
\beta_k(\delta_k-\delta_n)+(1-2^{n-k})\tau_k\lambda&=0.&& \label{eq:comm-auto-threshold}
\end{align}
Again, there are no further conditions for \(k<i<n\) in the reduced normal form.
\end{enumerate}

In particular, these equations determine a coefficient \(\beta_i\) only when the
corresponding factor \(\delta_i-\delta_n\) is nonzero.  If \(\delta_i=\delta_n\), the
same equation becomes a compatibility condition on the sink term and leaves
\(\beta_i\) free whenever that compatibility condition is satisfied.
\end{theorem}

\begin{proof}
For the annihilator-valued case and \(i<n\), one has
\[
R\phi({\bf e}_i)=(\delta_i\beta_i+\tau_i\lambda){\bf e}_n,
\qquad
\phi R({\bf e}_i)=\beta_i\delta_n{\bf e}_n.
\]
Thus \([R,\phi]({\bf e}_i)=0\) is equivalent to \eqref{eq:comm-auto-ann}; the basis
vector \({\bf e}_n\) gives no condition.

Assume now that \(R\) has nonzero threshold \(k\).  For \(i<k\) the same computation
gives \eqref{eq:comm-auto-before}.  At \(i=k\),
\[
R\phi({\bf e}_k)=2^{n-k}\delta_k\lambda{\bf e}_k+
(\delta_k\beta_k+\tau_k\lambda){\bf e}_n,
\]
whereas
\[
\phi R({\bf e}_k)=2^{n-k}\delta_k\lambda{\bf e}_k+
(2^{n-k}\tau_k\lambda+\beta_k\delta_n){\bf e}_n.
\]
Their equality is exactly \eqref{eq:comm-auto-threshold}.  Finally, for \(k<i<n\),
\[
[R,\phi]({\bf e}_i)=(1-2^{n-i})\tau_i\lambda{\bf e}_n={\bf 0},
\]
since \(\tau_i=0\) on the tail.  This proves the stated conditions and the final
assertion about resonances.
\end{proof}

\begin{definition}
For \(D\in\Der(\E)\) and \(\phi\in\Aut(\E)\), define
\[
\mathcal R_D=\{R\in\RB_0(\E):[R,D]=0\},
\qquad
\mathcal R_\phi=\{R\in\RB_0(\E):[R,\phi]=0\}.
\]
\end{definition}

\begin{proposition}\label{prop:nilpotent-derivation-commutant}
Let \(D\in\Der(\E)\) be of the form
\[
D({\bf e}_1)=\nu{\bf e}_n,
\qquad
D({\bf e}_i)={\bf 0}\quad(2\le i\le n).
\]
Then
\[
\mathcal R_D=
\begin{cases}
\RB_0(\E), & \nu=0,\\[1mm]
\{R\in\RB_0(\E):R({\bf e}_n)=0\}, & \nu\ne0.
\end{cases}
\]
In the second case, \(\mathcal R_D\) consists precisely of the annihilator-valued operators
\[
R({\bf e}_i)=\beta_i{\bf e}_n\quad(1\le i<n),
\qquad
R({\bf e}_n)=0.
\]
\end{proposition}

\begin{proof}
This is Theorem~\ref{thm:RB0-comm-der} with \(\mu=0\) when
\(\mathcal I_A=\varnothing\), and with the first case of that theorem when
\(\mathcal I_A\neq\varnothing\).
\end{proof}

\begin{proposition}
For \(D\) as in Proposition~\ref{prop:nilpotent-derivation-commutant}, one has
\[
\mathcal R_D=\mathcal R_{\exp(D)}.
\]
\end{proposition}

\begin{proof}
Here \(D^2=0\), so \(\exp(D)=\id+D\).  Therefore
\[
[R,\exp(D)]=[R,\id+D]=[R,D].
\]
\end{proof}

\subsection{Graph interpretation}

Let \(\Gamma(\E)\) be the directed graph with vertex set \(\{1,\dots,n\}\) and an edge
\(i\to j\) whenever the coefficient of \({\bf e}_j\) in \({\bf e}_i^2\) is nonzero.  In the
normal form \eqref{eq:natural-form}, the graph contains the chain
\[
1\to2\to\cdots\to n,
\]
possibly together with additional edges \(i\to j\) for \((i,j)\in\mathcal I_A\).  The
number \(i_0\) is the largest initial vertex of such an additional edge.  Thus the
threshold \(k\) of a nonzero ordinary or local Rota--Baxter operator must occur strictly
after all branching vertices encoded by \(\mathcal I_A\).

\begin{proposition}\label{prop:graph-interpretation}
Let \(R\in\RB_0(\E)\) have nonzero threshold \(k\).  Then for every interior tail vertex
\(i\in\{k+1,\dots,n-1\}\), every derivation \(D\in\Der(\E)\), and every automorphism
\(\phi\in\Aut(\E)\), one has
\[
[R,D]({\bf e}_i)={\bf 0},
\qquad
[R,\phi]({\bf e}_i)={\bf 0}.
\]
Consequently, in the reduced normal form the commutation problem is decided entirely
by the equations at the pre-threshold vertices \(i<k\) and at the threshold vertex \(i=k\).
\end{proposition}

\begin{proof}
For \(\mathcal I_A\neq\varnothing\), Proposition~\ref{prop:der-auto-forms} gives no
nonzero derivation sink terms on tail vertices and no automorphism sink terms except
possibly at \({\bf e}_1\).  For \(\mathcal I_A=\varnothing\), the same proposition gives
\(\tau_i=0\) for every \(i\ge2\).  Since an interior tail vertex satisfies \(i>k\ge1\),
Theorems~\ref{thm:RB0-comm-der} and~\ref{thm:RB0-comm-auto} give the displayed
vanishing.
\end{proof}

The graph therefore records two independent restrictions.  First, the Rota--Baxter
threshold cannot start before the last branching source \(i_0\).  Second, once the
threshold has been chosen, the geometric tail is already compatible with all reduced
normal-form derivations and automorphisms; only the vertices before the threshold and
the threshold vertex itself can impose coefficient or resonance conditions.

\section{Conclusion}

We have described local Rota--Baxter operators of weight $0$ on nilpotent evolution algebras of maximal nilindex. The classification shows that the diagonal rigidity known from the ordinary weight-zero theory remains present in the local setting: the nonzero diagonal coefficients, when they occur, must begin after the obstruction index $i_0$ and form a final geometric string determined by one scalar. The main difference is that the last-column coefficients become unrestricted. This explains precisely why the local class is generally larger than the ordinary class.

We also showed that the $s$-local version of the notion does not produce additional operators. Namely, every $s$-local Rota--Baxter operator of weight $0$ is already local. Moreover, the classification admits a natural parameter-space interpretation: $\LRB_0(\E)$ decomposes into an affine annihilator-valued stratum and quasi-affine strata corresponding to the first index at which the diagonal string begins. This gives the dimension comparison
\[
\dim \LRB_0(\E)=n+1,
\qquad
\dim \RB_0(\E)=n.
\]

The commutant results further show how ordinary weight-zero Rota--Baxter operators interact with derivations and automorphisms of $\E$. In the reduced normal form used in this paper, the only possible sink drift in derivations and automorphisms comes from the first basis vector. Consequently, the commutator conditions reduce to explicit threshold equations and, in the automorphism case, to spectral resonance conditions. The associated directed graph records why the threshold must occur after the last branching source and why the geometric tail is automatically compatible with the reduced normal-form derivations and automorphisms.

Several natural questions remain open. One may study local and $2$-local Rota--Baxter operators of nonzero weight on the same class of algebras, or extend the analysis to broader families of nilpotent evolution algebras where the graph structure is less rigid. Another possible direction is to investigate whether the commutant conditions obtained here can be related more systematically to graph invariants, derivation structure, or automorphism orbits of evolution algebras.

\section*{Statements and Declarations}

\subsection*{Funding.}
The second author (F.M.) acknowledges the UAEU UPAR Grant No.
G00004962. 

\subsection*{Author contributions} All authors contributed equally.

\subsection*{Data availability.}
No external datasets were used. The numerical values reported in the manuscript are generated by deterministic iteration and root-finding applied to the displayed cavity equations.

\subsection*{Competing interests.}
The author declares no competing interests.

\end{document}